\documentclass[11pt]{amsart}
\pdfoutput=1

\usepackage{xcolor}
\usepackage[utf8]{inputenc}
\usepackage[T1]{fontenc}
\usepackage[letterpaper,margin=1in]{geometry}
\usepackage{lmodern}
\usepackage{microtype}
\usepackage{etoolbox}
\usepackage{mathtools}
\usepackage{amssymb,bm}
\usepackage{mathrsfs}
\usepackage{array}
\usepackage{enumitem}
\usepackage[round]{natbib}
\usepackage{hyperref}

\numberwithin{equation}{section}

\patchcmd{\section}{\scshape}{\fontfamily{ppl}\bfseries\scshape\selectfont}{}{}

\hypersetup{colorlinks=true,linkcolor=blue!45!black,citecolor=blue!45!black,urlcolor=red!65!black}
\setlist{nosep,leftmargin=*}

\newtheorem{theorem}{Theorem}[section]
\newtheorem{prop}[theorem]{Proposition}
\newtheorem{lemma}[theorem]{Lemma}
\newtheorem{corr}[theorem]{Corollary}
\newtheorem{claim}[theorem]{Claim}

\DeclareMathOperator{\trace}{tr}
\DeclareMathOperator{\Diag}{Diag}
\DeclareMathOperator{\real}{Re}
\DeclareMathOperator{\imag}{Im}
\DeclareMathOperator{\per}{per}

\newcommand{\reals}{\mathbb{R}}
\newcommand{\C}{\mathbb{C}}
\newcommand{\Cc}{\mathscr{C}}
\newcommand{\half}{\tfrac{1}{2}}
\newcommand{\kron}{\otimes}
\newcommand{\inv}[1]{{#1}^{-1}}
\newcommand{\nlsum}{\sum\nolimits}
\newcommand{\nlprod}{\prod\nolimits}
\newcommand{\ip}[2]{\langle {#1},\, {#2} \rangle}
\newcommand{\norm}[1]{\|{#1}\|}
\newcommand{\dist}{\delta}
\newcommand{\posreal}{\reals_{\geqslant 0}}
\newcommand{\rdiv}{\rho}
\newcommand{\sdiv}{\delta_S}
\newcommand{\aper}{\per_\alpha}
\newcommand{\vm}{\bm{m}}
\newcommand\myhash{\texttt{\#}}

\hypersetup{
  pdftitle={Positive definite functions of noncommuting contractions, Hua-Bellman matrices, and a new distance metric},
  pdfauthor={Suvrit Sra}
}

\begin{document}
\title[Positive definite functions and Hua--Bellman matrices]{Positive definite functions of noncommuting contractions, Hua--Bellman matrices, and a new distance metric}
\author[Suvrit Sra]{Suvrit Sra\textsuperscript{\normalfont*}}
\thanks{*\,TU Munich, Dept.~of Mathematics, School of CIT; Garching, Germany.
  Email: \texttt{s.sra@tum.de}.}
\begin{abstract}
  We study positive definiteness of Hua--Bellman matrices (i.e., matrices of the form $[\det(I-A_i^*A_j)^{-\alpha}]_{ij}$, where $A_i$ and $A_j$ are strict contractions and $\alpha\in\reals$). We start by revisiting a 1959 work of \citeauthor{bellman1959representation} (\emph{Representation theorems and inequalities for Hermitian matrices}; Duke Mathematical J., 26(3), 1959) that studies positive definiteness of Hua--Bellman matrices and claims a strengthening of \citeauthor{hua1955inequalities}'s representation-theoretic results~(\emph{Inequalities involving determinants}; Acta Mathematica Sinica, 5 (1955)). We uncover a critical error in Bellman's proof that has surprisingly escaped notice to date, and we show by an explicit example that his claim itself is false. We then provide conditions under which $\det(I-A^*B)^{-\alpha}$ is a positive definite function. For complex contractions our condition is the known sharp one, with the integer exponents below Hua's range reached by an elementary argument. For real symmetric contractions we recover the part of Bellman's claims that survives. Building on our result, we introduce a new Poincar\'e-like distance metric on noncommuting strict contractions, and yet another `log-det' based distance in the appendix.
\end{abstract}
\maketitle

\section{Introduction}
We study a basic function on strictly contractive matrices, namely, $\det(I-Z)^{-\alpha}$, where $\|Z\| < 1$ and $\alpha \in \reals$. This function arises in a variety of contexts, including complex analysis~\citep{hua1955inequalities}, non-Euclidean geometry (e.g., Riemannian, Finsler, and K\"ahler geometry)~\citep{tian2000canonical,ssdiv,nielsen2020siegel}, optimization~\citep{sra2015conic}, combinatorics~\citep{goulden1992immanants,branden2012solutions,verejones1988generalization}, probability theory~\citep{shirai2007remarks}, matrix analysis~\citep{zhang2009recent,xu2009revisiting,ando2008positivity}, and mathematical physics~\citep{orlov2004solvable}, among others.

Specifically, we study conditions under which $(A,B) \mapsto \det(I-A^*B)^{-\alpha}$ is a \emph{positive definite function}, i.e., it can be written as an inner product $\ip{\Phi_\alpha(A)}{\Phi_\alpha(B)}$ for a suitable map $\Phi_\alpha$. We warn the reader that we use the term ``positive'' in its weak sense, e.g., we say a matrix $A$ is \emph{positive}, denoted $A\geqslant 0$, (\emph{strictly positive}, denoted $A>0$) if it is Hermitian and all its eigenvalues lie in $\posreal$ ($\reals_{>0}$). The term  ``positive definite'' should also thus be interpreted in the weak sense. We make this slight abuse of terminology to avoid focusing on merely technical conditions that ensure strict positivity but to rather focus on the more important task of ensuring positivity.

The study of positive definite functions has a long history in mathematics, dating back at least to \citet{szego1933potenzreihen}, who studied the closely related question about positivity of Taylor series coefficients of $P^{-\alpha}$ for the polynomial $P(x) = \sum_{i=1}^n\prod_{j\neq i}(1-x_j)$. \citet{scott2014complete} revisit this question for other polynomials such as the determinantal polynomial $P(x) = \det(\sum_i x_iA_i)$, and they study it via complete monotonicity; their results (which circulated as a preprint from 2011) form the basis of the work~\citep{branden2012solutions} that we will need later in this paper.

There is an additional, separate motivation behind our work. It lies in the following elegant identity of~\citet{hua1955inequalities}  for \emph{strict contractions} $A$ and $B$ (i.e., matrices whose largest singular value is strictly smaller than one):
\begin{equation}
  \label{eq:hua}
  I-B^*B + (A-B)^*(I-A^*A)^{-1}(A-B) = (I-A^*B)(I-AA^*)^{-1}(I-B^*A).
\end{equation}
Hua's identity~\eqref{eq:hua} immediately implies positivity of the block matrix
\begin{equation*}
  \begin{bmatrix}
    \inv{(I-A^*A)} & \inv{(I-A^*B)}\\
    \inv{(I-B^*A)} & \inv{(I-B^*B)}
  \end{bmatrix},
\end{equation*}
a property that is otherwise not obvious. Subsequent to Hua's discovery, various other works, e.g.,~\citep{marcus1958determinantal}, \citep{bellman1959representation}, \citep{ando1980hua}, ~\citep{xu2009revisiting,xu2011contractive}, and~\citep{zhang2009recent}, among others, continued the study of Hua-like block matrices and associated operator inequalities.
Beyond Hua's seminal work~\citep{hua1955inequalities}, a common  denominator of research on this topic has been~\citeauthor{bellman1959representation}'s (1959) work that studies Hua matrices through the lens of positive definite functions and integral representations.

Specifically, Bellman claims that for strictly contractive real matrices $A_1,\ldots,A_M$ of size $n\times n$, the \emph{Hua--Bellman} matrix
\begin{equation}
  \label{eq:1}
  H_\alpha := [\det(I-A_i^TA_j)^{-\alpha}]_{i,j=1}^M,
\end{equation}
is positive for $\alpha$ in the set $\{\frac{j}{2} \mid j=1,2,\ldots\}\cup (\half(n-1),\infty)$ (his Theorems~3 and~5). The positivity of $H_\alpha$ is tantamount to $\det(I-A^*B)^{-\alpha}$ being a positive definite function for the noted range of $\alpha$.
Bellman claims a significant generalization to \citeauthor{hua1955inequalities}'s work that had already established $H_\alpha \geqslant 0$ for $\alpha > n-1$; and seeks to prove his claim in a more elementary manner than Hua, whose results use representation theory and multivariable complex analysis.

While Bellman's approach is foundational, his proof contains an error that impacts subsequent works that rely on his claims. We uncover this old error, and study the positive definiteness of $\det(I-A^*B)^{-\alpha}$ afresh. Our main result on this topic is Theorem~\ref{thm:main}. As we will show, Bellman's claim is false, not merely his proof of it, and only its restriction to symmetric matrices survives. Subsequently, in Theorem~\ref{thm:dist} we present a closely related Poincar\'e-like geometry on noncommuting strict contractions that may be of independent interest.

\subsection{Uncovering an old error}
\label{sec:error}
\begin{claim}[\textbf{Theorem 3} in~\citep{bellman1959representation}]
  \label{claim:bellman}
  Let $A_1,\ldots,A_M$ be strictly contractive real matrices. Then, the Hua-matrix
  \begin{equation}
    \label{eq:bell}
    H_{k/2} = \left[\frac{1}{\det(I-A_i^TA_j)^{k/2}} \right]_{i,j=1}^M,
  \end{equation}
  is positive semidefinite for all integers $k\ge 1$.
\end{claim}
\noindent Bellman begins with the well-known Gaussian integral
$\frac{1}{\det(A)^{1/2}}  \propto \int_{\reals^m} e^{-x^TAx}dx$,
for a real symmetric positive definite matrix $A$. Then, on pg.~488,~(7.6), he invokes the following inequality of~\citet{ostrowski1951variation}:
\begin{equation}
  \label{eq:6}
  \det(I-A_i^TA_j) \ge \det(I - (A_i^TA_j)_s) > 0,
\end{equation}
where $X_s := \half(X+X^T)$ denotes the symmetric part of a matrix $X$. Subsequently, he states that (\emph{sic.} Theorem) Claim~\ref{claim:bellman} ``{\it
  \ldots will be demonstrated if we prove that $[\det(I - (A_i^TA_j)_s)^{-k/2}]$ is positive semidefinite.''}

\vskip6pt
\noindent This statement is where the error lies: just because inequality~\eqref{eq:6} holds, from positivity of $H_{k/2,s} := [\det(I - (A_i^TA_j)_s)^{-k/2}]$ we cannot conclude that $H_{k/2}$ is also positive definite, since $H_{k/2,s}$ only dominates $H_{k/2}$ entrywise. The claim itself is false, as the following example shows.

\vskip6pt
\noindent\textbf{Counterexample.} Let $n=2$, $k=1$, and $M=6$, and let
\begin{equation}
  \label{eq:cex}
  \begin{split}
    (A_j)_{j=1}^6=\biggl(
    \begin{bmatrix}
      -2 & -9\\
      -5 & -10
    \end{bmatrix},
    \begin{bmatrix}
      9 & -5\\
      9 & 6
    \end{bmatrix},
    &\begin{bmatrix}
      -10 & -3\\
      -6  & 3
    \end{bmatrix},
    \begin{bmatrix}
      -8 &-8\\
      1  &-10
    \end{bmatrix},
    \begin{bmatrix}
      -2 & 1\\
      -6 &-1
    \end{bmatrix},
    \begin{bmatrix}
      -1 & 3\\
      10 &-6
    \end{bmatrix}
    \biggr),\\
    A_j &\gets \frac12 \frac{A_j}{\norm{A_j}}\quad\text{for}\ 1\le j\le 6,
  \end{split}
\end{equation}
so that each $A_j$ is a real strict contraction of operator norm $\frac12$. The $6\times 6$ matrix $H_{1/2}=[\det(I-A_i^TA_j)^{-1/2}]_{i,j=1}^6$ has smallest eigenvalue
\begin{equation*}
  \lambda_{\min}(H_{1/2}) \approx -1.2066\times 10^{-3},
\end{equation*}
so $H_{1/2}$ is not positive.\footnote{The value was confirmed in $60$-digit arithmetic. Every $5\times 5$ principal submatrix of this $H_{1/2}$ is positive.} Bellman's claim thus fails already for $2\times 2$ real matrices and $k=1$.

\vskip6pt
\noindent What, then, survives of Bellman's claim? We show that $H_\alpha\geqslant 0$ for every integer $\alpha\geqslant 0$ and for complex contractions of any size (Theorem~\ref{thm:main}-(i)); together with Hua's range $\alpha>n-1$ this is the complete set of admissible exponents for complex contractions~\citep{faraut1990function}, and our proof of the integer part, via $\alpha$-permanents, is elementary. We also show that half-integer exponents are admissible when the $A_i$ are real and symmetric (Theorem~\ref{thm:main}-(ii)), which is the symmetric case of Bellman's claims.

\section{Positive definite functions on noncommuting contractions}
In this section we present results characterizing the positive definiteness of Hua--Bellman matrices. To that end, we study the corresponding function $\det(I-A^*B)^{-\alpha}$ on $A, B \in \Cc_n$, the class of $n\times n$ strictly contractive (in operator norm) complex matrices.\footnote{To reduce clutter, we limit ourselves to square matrices; similar arguments can be made when $A$, $B$ are $n \times k$ strict contractions.} We provide sufficient conditions on $\alpha$ to ensure the desired positive definiteness.

We are now ready to state the first main result of this paper.
\begin{theorem}
  \label{thm:main}
  Let $A, B \in \Cc_n$. Then, $\det(I-A^*B)^{-\alpha}$ is a positive definite function:
  \begin{enumerate}[label=(\roman*)]
    \setlength{\itemsep}{2pt}
  \item for $\alpha \in \{0,1,2,\ldots\}\cup (n-1,\infty)$; and
  \item for $\alpha \in \{\frac{j}{2}\mid j=0,1,2,\ldots\}\cup (\frac{n-1}{2},\infty)$ if $A$, $B$ are real and symmetric.
  \end{enumerate}
\end{theorem}
Part~(i) contains Hua's result ($\alpha>n-1$) and adds the integers $0,1,\ldots,n-1$. The set in part~(i) is the Wallach set of the matrix ball, and by \citet{faraut1990function} (see also \citep[Sec.~XIII.2]{faraut1994analysis}) it is exactly the set of $\alpha$ for which $\det(I-A^*B)^{-\alpha}$ is positive definite on $\Cc_n$; so part~(i) is sharp, and the point of the proof below is that the integers $0,1,\ldots,n-1$ are reached by an elementary argument. Part~(ii) covers Theorem~5 of Bellman's paper, albeit with an additional requirement of symmetry.\footnote{We present this result since the error in the proof of Claim~\ref{claim:bellman} also impacts Theorem 5 of Bellman, wherein he claims that \emph{``the proof proceeds along the same lines...''}.} We note in passing that the set of valid $\alpha$ noted in part-(ii) is precisely the Gindikin--Wallach set---see e.g., \citep{scott2014complete} and \citep[pg.~42]{belton2022preservers} and references therein for additional connections of this set.
Bellman closes his Theorem~5 with the words ``what the precise exponent should be is not clear''. For real symmetric contractions the Gindikin set is the answer: by Gindikin's theorem~\citep{gindikin1964analysis} (see also \citep[Ch.~VII]{faraut1994analysis}) it is exactly the set of $\alpha$ for which $\det(X+Y)^{-\alpha}$ is a positive definite kernel on the cone of positive definite real symmetric matrices, and the Cayley transform in the proof of part~(ii) below shows that $\det(I-A^*B)^{-\alpha}$ on real symmetric contractions is positive definite for exactly the same $\alpha$; so part~(ii) is sharp as well.

\vskip6pt
\noindent The key ingredient in our proof of Theorem~\ref{thm:main} is the $\alpha$-permanent, which we now recall.

\subsection{The \texorpdfstring{$\alpha$}{alpha}-permanent}
The $\alpha$-permanent generalizes the matrix permanent and determinant, and was introduced by~\citet{verejones1988generalization}. Let $\alpha\in \C$ and $A=(a_{ij})$ be an $n\times n$ matrix. Then, the \emph{$\alpha$-permanent} is defined as
\begin{equation}
  \label{eq:2}
  \aper(A) := \sum_{\sigma \in \mathfrak{S}_n} \alpha^{\myhash\sigma}\prod_{i=1}^n a_{i,\sigma(i)},
\end{equation}
where $\myhash\sigma$ denotes the number of disjoint cycles in the permutation $\sigma$. This object interpolates between the determinant and permanent (indeed, $\per_1(A)=\per(A)$, while $\per_{-1}(A)=(-1)^n\det(A)$), and it enjoys a variety of applications and connections; see e.g.,~\citep{shirai2007remarks,crane2013algebraic,frenkel2009remarks}. The key property of $\alpha$-permanents that we will need is an ``inner-product'' representation derived in Lemma~\ref{lem:perip}.
\begin{lemma}
  \label{lem:perip}
  Let $A$, $B \in \mathbb{C}^{n\times m}$, $\alpha \in \reals$, and $n,m \ge 1$. Then, there exist scalars  $\{c_\lambda^\alpha\}_{\lambda \vdash m}$ and nonlinear maps $\{\psi_\lambda\}_{\lambda \vdash m}: \mathbb{C}^{n\times m} \to \mathbb{C}^{n_\lambda \times m_\lambda}$ for suitable integers $\{n_\lambda, m_\lambda\}_{\lambda \vdash m}$ ($\lambda \vdash m$ signifies that $\lambda$ is a partition of $m$), such that the identity
  $$\aper(A^*B) = \sum_{\lambda\vdash m} c_\lambda^\alpha\ip{\psi_\lambda(A)}{\psi_\lambda(B)},$$ holds, where for each $\lambda$, $\ip{\cdot}{\cdot}$ is the associated trace inner product on  $\mathbb{C}^{n_\lambda \times m_\lambda}$.
\end{lemma}
\begin{proof}
  The proof relies on the (known) observation that the $\alpha$-permanent can be written in terms of immanants. The \emph{immanant} of $X\in \mathbb{C}^{m\times m}$ indexed by a partition $\lambda$ of $m$ is defined as (see~\citep{merris1997multilinear} for details)
  \begin{equation}
    \label{eq:25}
    d_\lambda(X) := \nlsum_{\sigma \in \mathfrak{S}_m} \chi_\lambda(\sigma) \nlprod_{i=1}^m x_{i,\sigma(i)},
  \end{equation}
  where $\chi_\lambda(\sigma)$ is the character associated to the irreducible representation $S^\lambda$, the Specht module corresponding to $\lambda \vdash m$---see e.g.,~\citep{fulton1991representation}. Then, \citet[Eq.~(12)]{crane2013algebraic} shows that for each $\lambda$, there exist constants $c_\lambda^\alpha$ such that:\footnote{\citep[Theorem~2.4]{crane2013algebraic} notes that $c_\lambda^\alpha = \frac{1}{m!}\sum_{\sigma \in \mathfrak{S}_m} \alpha^{\myhash\sigma}\chi_\lambda(\sigma)$.}
  \begin{equation}
    \label{eq:26}
  \per_\alpha(X) = \nlsum_{\lambda \vdash m} c_\lambda^\alpha d_\lambda(X).
  \end{equation}
  From classic results, e.g.,~\citep[Thm.~7.26]{merris1997multilinear}, it follows that for each $\lambda \vdash m$, there exists a matrix $P_\lambda$ with $m^m$ rows and $m_\lambda$ columns such that $d_\lambda(X)=\trace P_\lambda^* X^{\kron m}P_\lambda$. Thus (with $n_\lambda=n^m$),
  \begin{equation}
    \label{eq:27}
    d_\lambda(A^*B) = \trace P_\lambda^*(A^*B)^{\kron m}P_\lambda = \trace P_\lambda^*(A^{\otimes m})^*(B^{\otimes m})P_\lambda =: \ip{\psi_\lambda(A)}{\psi_\lambda(B)},
  \end{equation}
  where $\psi_\lambda: A \mapsto A^{\otimes m}P_\lambda$. Using representation~\eqref{eq:27} in~\eqref{eq:26} we obtain the identity
  \begin{align*}
    \per_\alpha(A^*B) = \nlsum_{\lambda \vdash m} c_\lambda^\alpha \ip{\psi_\lambda(A)}{\psi_\lambda(B)},
  \end{align*}
  which shows that $\aper(A^*B)$ is indeed a weighted sum of inner products.
\end{proof}

While Lemma~\ref{lem:perip} shows that $\aper(A^*B)$ can be written as a weighted sum of inner products, the coefficients $c_\lambda^\alpha$ can be negative, in which case the sum need not be a positive definite kernel. For nonnegative integers $\alpha$ they are not.
\begin{prop}
  \label{prop:branden2}
  Let $k\in\mathbb{Z}_{\ge 0}$ and $m\ge 1$. Then $c_\lambda^k\geqslant 0$ for every $\lambda\vdash m$; indeed, $c_\lambda^k$ is the multiplicity of the Specht module $S^\lambda$ in the representation of $\mathfrak{S}_m$ on $(\mathbb{C}^k)^{\otimes m}$ that permutes the tensor factors.
\end{prop}
\begin{proof}
  A permutation $\sigma$ fixes the basis tensor $e_{i_1}\otimes\cdots\otimes e_{i_m}$ iff $(i_1,\ldots,i_m)$ is constant on the cycles of $\sigma$, so the character of this representation is $\sigma\mapsto k^{\myhash\sigma}$. Hence $c_\lambda^k=\frac{1}{m!}\sum_{\sigma\in\mathfrak{S}_m}k^{\myhash\sigma}\chi_\lambda(\sigma)$ is the inner product of this character with $\chi_\lambda$, i.e., the multiplicity of $S^\lambda$, a nonnegative integer. (By Schur--Weyl duality it equals the number of semistandard Young tableaux of shape $\lambda$ with entries in $\{1,\ldots,k\}$; see~\citep{fulton1991representation}.)
\end{proof}
\noindent For non-integer $\alpha$ the coefficients can be negative: for $m=2$ one has $\per_\alpha=\frac{\alpha^2+\alpha}{2}\per+\frac{\alpha^2-\alpha}{2}\det$, so $c_{(1,1)}^{1/2}=-\frac18$. \citet{branden2012solutions}, building on the complete monotonicity results of~\citet{scott2014complete} (then a preprint), shows that nevertheless $\per_\alpha(C)\geqslant 0$ for every real positive matrix $C$ when $\alpha$ is a half-integer (his Theorem~2.3 is stated for the $\alpha$-determinant $\det_\alpha(C):=\alpha^n\per_{1/\alpha}(C)$). For the kernel $(A,B)\mapsto\per_\alpha(A^TB)$ this is only the one-point condition $\per_\alpha(A^TA)\geqslant 0$, which is weaker than positive definiteness; the counterexample of Section~\ref{sec:error} shows that the difference is real.

In addition to Lemma~\ref{lem:perip} and Proposition~\ref{prop:branden2}, our proof of Theorem~\ref{thm:main} uses the following generalization of MacMahon's Master Theorem.
\begin{theorem}[\citep{verejones1988generalization,foata1988laguerre}]
  \label{thm:mac}
  Let $A=[a_{ij}]_{i,j=1}^n$ and let $\vm=(m_1,\ldots,m_n) \in \mathbb{N}^n$, where $\mathbb{N}=\{0,1,2,\ldots\}$. Let $A[\vm]$ be the $|\vm|\times |\vm|$ matrix with $|\vm|=\nlsum_{j=1}^n m_j$, obtained by replacing the $(i,j)$-entry of $A$ by the $m_i\times m_j$ matrix $a_{ij}\bm{1}_{m_i}\bm{1}^T_{m_j}$. Let $X=\Diag(x_1,\ldots,x_n)$ and $\alpha \in \C$. Then,
  \begin{equation}
    \label{eq:3}
    \det(I - XA)^{-\alpha} = \sum_{\vm \in \mathbb{N}^n}\frac{\bm{x}^{\vm}}{\vm!}\per_\alpha(A[\vm]),
  \end{equation}
  where $\bm{x}^{\vm}=x_1^{m_1}\cdots x_n^{m_n}$ and $\vm!=m_1!\cdots m_n!$.
\end{theorem}

\subsection{Proof of Theorem~\ref{thm:main}, part (i).}
First, recall from \citep{hua1955inequalities} that if $\alpha > n-1$, then $\det(I-A^*B)^{-\alpha}$ is a positive definite function\footnote{\citet{bellman1959representation} recalls this result of Hua as Theorem~2 of his paper.}, a result that holds for complex contractions (and thus \emph{a fortiori} also for real ones). What remains to prove is that the integers $\alpha=0,1,\ldots,n-1$ are admissible (the case $\alpha=0$ being trivial). The main idea is to invoke Theorem~\ref{thm:mac} and use it to represent $\det(I-A^*B)^{-\alpha}$ as an inner product $\ip{\Phi_\alpha(A)}{\Phi_\alpha(B)}$, where $\Phi_\alpha$ maps into a suitable Hilbert space. As a consequence, we obtain that the associated Hua--Bellman matrix $H_\alpha$ is a Gram matrix, and hence positive.

\vskip5pt
\noindent The first step is to realize that $A[\vm] = Q_{\vm}^*(A\otimes \bm{11}^T)Q_{\vm}$ for matrices $Q_{\vm}$ and $\bm{11}^T$ of appropriate sizes induced by Theorem~\ref{thm:mac}. Let $E_{ij}$ denote the $m_i \times m_j$ matrix of all ones, and $\bm{11}^T=E=[E_{ij}]_{i,j=1}^n$ the corresponding block matrix of size $|\bm{m}|\times|\bm{m}|$.\footnote{To reduce notational burden, we let the dependency of blocks of $E$ on $\vm$ to remain implicit.} Then, consider the Khatri--Rao product between conformally partitioned matrices $A$ and $E$:
\begin{equation}
  \label{eq:4}
  A \ast E =
  \begin{bmatrix}
    a_{11}\kron E_{11} & \cdots  & a_{1n}\kron E_{1n}\\
    \vdots & \ddots & \vdots \\
    a_{n1} \kron E_{n1} & \cdots & a_{nn}\kron E_{nn}
  \end{bmatrix},
\end{equation}
which is nothing but $A[\vm]$ since $a_{ij}\kron E_{ij} = a_{ij}E_{ij}$. From basic properties of Khatri--Rao products~\citep{liu2008hadamard} it follows that there is a matrix $Q_{\vm}$ such that $A[\vm]=Q_{\vm}^*(A\kron E)Q_{\vm}$. More explicitly, let $Q_{\vm}$ be the diagonal matrix $\Diag(U^{\vm}_1,\ldots,U^{\vm}_n)$, where each $U^{\vm}_i$ is an $|\vm|\times m_i$ matrix formed by suitable columns of the identity matrix $I_{|\vm|}$, such that $(U^{\vm}_i)^*EU^{\vm}_j=E_{ij}$. Therewith, we have the factorization
\begin{small}
  \begin{align}
    \nonumber
    Q_{\vm}^*(A\kron E)Q_{\vm} &= \begin{bmatrix}
      (U^{\vm}_1)^* && \\
      & \ddots & \\
      && (U^{\vm}_n)^*
    \end{bmatrix}
    \begin{bmatrix}
      a_{11}E & \cdots  & a_{1n}E\\
      \vdots & \ddots & \vdots \\
      a_{n1}E & \cdots & a_{nn}E
    \end{bmatrix}
    \begin{bmatrix}
      U^{\vm}_1 && \\
      & \ddots & \\
      && U^{\vm}_n
    \end{bmatrix}\\
    \label{eq:5}
    &=
    \begin{bmatrix}
      a_{11}(U^{\vm}_1)^*EU^{\vm}_1 & a_{12}(U^{\vm}_1)^*EU^{\vm}_2 & \cdots & a_{1n}(U^{\vm}_1)^*EU^{\vm}_n\\
      \vdots & \vdots & & \vdots\\
      a_{n1}(U^{\vm}_n)^*EU^{\vm}_1 & a_{n2}(U^{\vm}_n)^*EU^{\vm}_2 & \cdots & a_{nn}(U^{\vm}_n)^*EU^{\vm}_n
    \end{bmatrix}\\
    \nonumber
    &=\begin{bmatrix}
      a_{11}E_{11} & \cdots  & a_{1n}E_{1n}\\
      \vdots & \ddots & \vdots \\
      a_{n1}E_{n1} & \cdots & a_{nn}E_{nn}
    \end{bmatrix}=A[\vm].
  \end{align}
\end{small}%
From identity~\eqref{eq:3} we have
$\det(I-XA^*B)^{-\alpha} = \sum_{\vm \in \mathbb{N}^n}\frac{\bm{x}^{\vm}}{\vm!}\per_\alpha\bigl( (A^*B)[\vm]\bigr)$, so that using factorization~\eqref{eq:5} we can write
\begin{equation}
  \label{eq:7}
  (A^*B)[\vm] = Q_{\vm}^*( A^*B \kron \bm{11}^T)Q_{\vm} = Q_{\vm}^*((A\otimes 1^T)^*(B\otimes 1^T))Q_{\vm} =: \tilde{A}_{\vm}^*\tilde{B}_{\vm}.
\end{equation}
Finally, identity~\eqref{eq:7} combined with Lemma~\ref{lem:perip} immediately implies that
\begin{equation}
  \label{eq:28}
  \aper((A^*B)[\vm]) = \nlsum_{\lambda \vdash |\vm|} c_\lambda^\alpha\ip{\psi_\lambda(\tilde{A}_{\vm})}{\psi_\lambda(\tilde{B}_{\vm})},
\end{equation}
for suitable nonlinear maps $\{\psi_\lambda\}_{\lambda \vdash |\bm{m}|}$. Each summand $\ip{\psi_\lambda(\tilde{A}_{\vm})}{\psi_\lambda(\tilde{B}_{\vm})}$ is a Gram kernel, hence positive definite, and by Proposition~\ref{prop:branden2} the coefficients $c_\lambda^\alpha$ are nonnegative for integer $\alpha$. Thus~\eqref{eq:28} is a positive definite kernel in $(A,B)$.

Now invoking Theorem~\ref{thm:mac}, we see that $\det(I-A^*B)^{-\alpha}$ is a nonnegative sum of positive definite kernels. Consequently, there exists a suitable nonlinear map $\Phi_\alpha$ (that maps its argument to a suitable Hilbert space) so that $\det(I-A^*B)^{-\alpha}=\ip{\Phi_\alpha(A)}{\Phi_\alpha(B)}$, concluding the proof.\qed

\vskip6pt
\noindent\textbf{Remark.} For integer $\alpha$ there is also a direct route that avoids $\alpha$-permanents: since $\det(I-Z)^{-1}=\sum_{k\geqslant 0}\trace\,\mathrm{Sym}^k(Z)$ for $\|Z\|<1$ and $\mathrm{Sym}^k(A^*B)=\mathrm{Sym}^k(A)^*\mathrm{Sym}^k(B)$, the kernel $\det(I-A^*B)^{-1}=\sum_k\ip{\mathrm{Sym}^k(A)}{\mathrm{Sym}^k(B)}$ is positive definite, and its integer powers are positive definite by the Schur product theorem. The $\alpha$-permanent expansion is kept because it exhibits the structure of the coefficients (Proposition~\ref{prop:branden2}) and shows exactly where the half-integer case breaks down.

\vskip6pt
It remains to prove Theorem~\ref{thm:main}-(ii). To that, we first recall a closely related result.
\begin{theorem}[Gindikin~\citep{gindikin1964analysis}; stated in this form in \protect{\citep[Theorem~3.10]{ssdiv}}]
  \label{thm:kernel}
  Let $A_1,\ldots,A_M\in \reals^{n\times n}$ be symmetric and strictly positive. Then, the matrix
  \begin{equation}
    \label{eq:9}
    \left[\det(A_i+A_j)^{-\alpha}\right]_{i,j=1}^M,
  \end{equation}
  is positive whenever $\alpha \in \{\frac{j}{2}\mid j=0,1,2,\ldots,\} \cup (\frac{n-1}{2},\infty)$.
\end{theorem}

\subsection{Proof of Theorem~\ref{thm:main}, part (ii).}
The key idea is to introduce a change of variables that turns $\det(I-A^*B)^{-\alpha}$ essentially into the form~\eqref{eq:9}. Specifically, consider the \emph{Cayley transformation}  $X=(I-A)(I+A)^{-1}$ (similarly, $Y=(I-B)(I+B)^{-1}$). Matrix $X$ (similarly, $Y$) is well-defined because $I+A$ is invertible due to $A$ being a strict contraction. The reverse transformation is $X \mapsto (I-X)(I+X)^{-1}$, which is also well-defined because $I+X = 2I(I+A)^{-1}$ is also invertible. Using this transformation on $A$ and $B$ we obtain
\begin{align*}
  I-A^*B &= I - (I+X^*)^{-1}(I-X^*)(I-Y)(Y+I)^{-1}\\
  &= I - (I+X^*)^{-1}[I+X^*Y-X^*-Y](Y+I)^{-1}\\
  &= (I+X^*)^{-1}\bigl[(I+X^*)(Y+I) - X^*Y+X^*+Y-I\bigr](Y+I)^{-1}\\
  &= (I+X^*)^{-1}[2X^*+2Y](Y+I)^{-1}.
\end{align*}
Next, recall the notation $\real(Z)=\frac{Z+Z^*}{2}$ and $\imag(Z) = \frac{Z-Z^*}{2i}$ (where $i=\sqrt{-1}$), so that any complex matrix $Z$ can be written as $Z = \real(Z)+i\imag(Z)$. Since $A \in \Cc_n$, $I-A^*A=(I+X^*)^{-1}[2X^*+2X](I+X)^{-1}$ is strictly positive. Thus,
\begin{equation*}
  \real(X) = \tfrac14(I+X^*)(I-A^*A)(I+X),
\end{equation*}
which is clearly strictly positive. If $A$ and $B$ are Hermitian, then $\real(X)=X=X^* > 0$ (similarly, $Y>0$), whereby
\begin{equation}
  \label{eq:10}
  \det(I-A^*B)^{-\alpha} = \frac{\det(I+X)^\alpha\det(I+Y)^\alpha}{\det(2(X+Y))^\alpha},
\end{equation}
whose positivity  for the range of values claimed in Theorem~\ref{thm:main}-(ii) immediately follows from Theorem~\ref{thm:kernel}, once we additionally assume $A$ and $B$ to be real.

\section{A Poincar\'e-like geometry on noncommuting contractions}
\label{sec:metric}
In this section, we introduce a new Poincar\'e-like geometry on the space of strict contractions. The distance function introduced is suggested by Hua--Bellman matrices, whose positivity yields the positivity of the proposed distance. The second main result of this paper is Theorem~\ref{thm:dist}, which formally introduces the said geometry. We state the results for the set of $n\times k$ strict contractions, denoted $\Cc_{n,k}$.
\begin{theorem}
  \label{thm:dist}
  Define $\delta: \Cc_{n,k} \times \Cc_{n,k} \to \reals_{\geqslant 0}$ as
  \begin{equation}
    \label{eq:dist}
    \dist^2(A, B) :=
    \log \frac{|\det(I-A^*B)|}{\sqrt{\det(I-A^*A)}\sqrt{\det(I-B^*B)}},\qquad\text{for}\ A, B \in \Cc_{n,k}.
  \end{equation}
  Then, $(\Cc_{n,k},\delta)$ is a metric space.
\end{theorem}
It is worth noting that positivity of an $M\times M$ Hua--Bellman matrix depends only on the dimensionality $n$ of its constituent matrices and not on $M$, and thus, Theorem~\ref{thm:main} holds for arbitrary $M$.  In contrast, the distance function of Theorem~\ref{thm:dist} corresponds to an absolute value version of $H_\alpha$ that holds for $M=3$, and for which a much wider range of values of $\alpha$ is admissible. More precisely, we have the following corollary:
\begin{corr}
  \label{cor:dist}
  Let $A$, $B$, and $C \in \Cc_{n,k}$. Let $\alpha \geqslant 0$. Then, the following matrix
  \begin{equation}
    \label{eq:8}
    H_{\frac\alpha2} \circ \overline{H}_{\frac\alpha2} =
    \begin{bmatrix}
      \frac{1}{\det(I-A^*A)^\alpha} & \frac{1}{|\det(I-A^*B)|^\alpha} & \frac{1}{|\det(I-A^*C)|^\alpha}\\
      \frac{1}{|\det(I-B^*A)|^\alpha} & \frac{1}{\det(I-B^*B)^\alpha} & \frac{1}{|\det(I-B^*C)|^\alpha}\\
      \frac{1}{|\det(I-C^*A)|^\alpha} & \frac{1}{|\det(I-C^*B)|^\alpha} & \frac{1}{\det(I-C^*C)^\alpha}
    \end{bmatrix},
  \end{equation}
  is positive, where $H_\alpha$ denotes a Hua--Bellman matrix as before.
\end{corr}
\begin{proof}
  Since $\delta^2$ is the square of a distance, we know from classical results (see e.g., the exposition in~\citep[App.~A]{sra2021metrics}) that the matrix
  \begin{equation*}
    M := \begin{bmatrix}
      0 & \delta^2(A,B) & \delta^2(A,C)\\
      \delta^2(A,B) & 0 & \delta^2(B,C)\\
      \delta^2(A,C) & \delta^2(B,C) & 0
    \end{bmatrix}
  \end{equation*}
  is conditionally negative definite. Hence, its elementwise exponentiation $M' := [e^{-\alpha m_{ij}}]_{i,j=1}^3$ is positive (i.e., $M' \geqslant 0$) for all $\alpha \geqslant 0$---see e.g., \citep[Thm.~3.2.2]{berg1984}. We conclude the proof by introducing the diagonal matrix
  \begin{equation*}
   D =
   \begin{bmatrix}
     \tfrac{1}{\det(I-A^*A)^{\frac\alpha2}} &&\\
     & \tfrac{1}{\det(I-B^*B)^{\frac\alpha2}} & \\
     &  & \tfrac{1}{\det(I-C^*C)^{\frac\alpha2}}
   \end{bmatrix}
  \end{equation*}
  and observing that $H_{\frac\alpha2}\circ\overline{H}_{\frac\alpha2}=DM'D \geqslant 0$.
\end{proof}
\vskip8pt
\noindent\textbf{Remark.} The range $\alpha \geqslant 0$ proved in Corollary~\ref{cor:dist} is non-trivial. Indeed, from Theorem~\ref{thm:main} it follows that $H_1\geqslant 0$. Since here $H_1$ is $3\times 3$, an observation of \citet{marcus1971partitioned} shows that its elementwise absolute value $|H_1|$ is again positive (this does not follow from Theorem~\ref{thm:main}, as that would require $H_{1/2} \geqslant 0$). Consequently, from a well-known result of~\citet{fitzgerald1977fractional} on Hadamard powers, since $M=3$, we can only conclude positive definiteness of the Hadamard power $|H_1|^{\circ \alpha}$ for $\alpha \geqslant M-2=1$, whereas Corollary~\ref{cor:dist} permits $\alpha \geqslant 0$.

\vskip6pt
Corollary~\ref{cor:hyper} explains why we call the geometry induced by distance~\eqref{eq:dist} Poincar\'e-like.
\begin{corr}
  \label{cor:hyper}
  Let $A, B \in \Cc_{n,1}$, i.e., vectors of Euclidean norm less than one. Then
  \begin{equation*}
    \dist^2(A,B) = \log \frac{|1-A^*B|}{\sqrt{(1-\|A\|^2)}\sqrt{(1-\|B\|^2)}}.
  \end{equation*}
  For real vectors this equals $\log\cosh d(A,B)$, where $d$ is the Cayley--Klein--Hilbert (Beltrami--Klein) distance of the unit ball~\citep[pg.~120]{deza2013encyclopedia}; for complex vectors it equals $\log\cosh\bigl(\tfrac12 d_P(A,B)\bigr)$, where $d_P$ is the Poincar\'e--Bergman distance of the complex unit ball, normalized so that for $n=1$ it is the Poincar\'e distance of the disk.
\end{corr}

\vskip6pt
Before proving Theorem~\ref{thm:dist}, we first recall a closely related distance function.
\begin{theorem}[S-Divergence~\citep{ssdiv}]
  \label{thm:sdiv}
  Let $X, Y$ be strictly positive definite. Then,
  \begin{equation}
    \label{eq:17}
    \sdiv^2(X,Y) := \log\frac{\det\left(\tfrac{X+Y}{2}\right)}{\sqrt{\det(X)}\sqrt{\det(Y)}},
  \end{equation}
  is the square of a distance, i.e., $\sdiv$ is a distance.
\end{theorem}

At this point, one may wonder whether we can use the Cayley transformation used in the proof of Theorem~\ref{thm:main}-(ii) to approach Theorem~\ref{thm:dist}. Directly applying this transformation to $\delta^2(A,B)$ leads to the following object:
\begin{equation}
  \label{eq:12}
  \log\frac{|\det(X^*+Y)|}{\sqrt{\det(X^*+X)}\sqrt{\det(Y^*+Y)}}.
\end{equation}
This quantity bears a close resemblance to the S-Divergence~\eqref{eq:17}, yet it does not seem amenable to the proof technique developed in~\citep{ssdiv} for proving the triangle inequality for the $S$-Divergence.\footnote{Indeed, in the initial version of this article, there was an error that arose from following that approach.} On the other hand, if $A$ and $B$ were Hermitian, then identity~\eqref{eq:10} would apply, and the quantity~\eqref{eq:12} would reduce to the S-Divergence (note $\real(X)$, $\real(Y)$ are strictly positive definite). This observation was crucial to the discovery of our proof of Theorem~\ref{thm:dist} below.

\subsection{Proof of Theorem~\ref{thm:dist}}
The key idea is to consider the \emph{Hermitian dilations} of $A$ and $B$, i.e.,
\begin{equation}
  \label{eq:13}
  P :=\begin{bmatrix}
    0 & A^*\\
    A & 0
  \end{bmatrix},\quad
  Q :=\begin{bmatrix}
    0 & B^*\\
    B & 0
  \end{bmatrix}.
\end{equation}
Then, computing $\det(I-P^*Q)$ we see that
\begin{equation}
  \label{eq:14}
  \det(I-P^*Q) = \det\left(\begin{bmatrix}I- A^*B& 0\\ 0 & I-AB^*\end{bmatrix}\right)
  = \det(I-A^*B)\det(I-AB^*).
\end{equation}
But $\det(I-AB^*)=\det(I-B^*A)$, so that $\det(I-P^*Q)=|\det(I-A^*B)|^2$. Consequently,
\begin{equation}
  \label{eq:29}
  \frac{|\det(I-A^*B)|}{\sqrt{\det(I-A^*A)\det(I-B^*B)}}
  =
  \frac{\sqrt{\det(I-P^*Q)}}{\sqrt[4]{\det(I-P^*P)\det(I-Q^*Q)}}.
\end{equation}
From identity~\eqref{eq:29} upon taking logarithms it follows that
\begin{equation}
  \label{eq:30}
  \delta^2(A,B) = \half\delta^2(P,Q).
\end{equation}
Since $P$ and $Q$ are Hermitian strict contractions, as argued above, we can reduce $\delta^2(P,Q)$ to a corresponding S-Divergence, whereby upon invoking Theorem~\ref{thm:sdiv}, we immediately conclude that $\delta$ is a distance metric.

\vskip6pt
\emph{Remarks:} We believe that the distance function~(\ref{eq:dist}) may find several applications. Our belief is based on the diverse body of applications the S-Divergence~(\ref{eq:17}) enjoys, e.g., in computer vision~\citep{cherian2012jensen}, brain-computer interfaces and imaging~\citep{yger2016riemannian}, matrix means~\citep{ssdiv,chebbi2012means}, geometric optimization~\citep{sra2015conic}, numerical linear algebra~\citep{sra2016square}, signal processing~\citep{bouchard2018riemannian}, machine learning~\citep{zern2018unsupervised,tiomoko2019random}, quantum information theory~\citep{virosztek2021metric}, among many others.

\vskip8pt
\section*{Acknowledgements}
We thank the referee for their exceptional review of an initial version of this paper, including uncovering an error in our initial proof (which we have now fixed). We also acknowledge support from an NSF CAREER grant (IIS-1846088), and NSF CCF-2112665 (TILOS AI Research Institute). The author (SS) thanks the Alexander von Humboldt (AvH) Foundation for their generous support via a Humboldt Professorship in AI.

\appendix
\section{A related distance on arbitrary complex matrices}
\label{sec:appendix}
An earlier version of this article proved the triangle inequality for $\delta$ by first applying the Cayley transformation and then invoking the theorem below. That route turned out to be flawed (see the footnote following~\eqref{eq:12}), and the Hermitian dilation argument of Section~\ref{sec:metric} supersedes it. Theorem~\ref{thm:adiv} is nevertheless valid on its own, and since it seems to be of independent interest, we record it here; it does not play any role in the results of Sections~1--3.
\begin{theorem}
  \label{thm:adiv}
  Let $X, Y$ be arbitrary complex matrices, and $0< p \le 2$, then
  \begin{equation}
    \label{eq:18}
    \rdiv_p^2(X,Y) := \log\det(I + |X-Y|^p),
  \end{equation}
  is the square of a distance, i.e., $\rdiv_p$ is a distance (here $|X| := (X^*X)^{1/2}$).
\end{theorem}
\noindent We will need the following simple observation in our proof of Theorem~\ref{thm:adiv}:
\begin{lemma}
  \label{lem:cve} Let $0< p \le 2$. The function $f(t)=\sqrt{\log(1+t^p)}$ is concave on $(0,\infty)$.
\end{lemma}
\begin{proof}
  We prove that $f''(t) \le 0$. Since $f''(t)=-\frac{p t^{p-2} \left(p t^p+2 \left(t^p-p+1\right) \log \left(t^p+1\right)\right)}{4 \left(t^p+1\right)^2 \log ^{\frac{3}{2}}\left(t^p+1\right)}$, it suffices to verify that $\left(p t^p+2 \left(t^p-p+1\right) \log \left(t^p+1\right)\right) \ge 0$. Writing $t^p=x$, this inequality is equivalent to
  $p x + 2 (1+x)\log(1+x) \ge 2p\log(1+x)$. Since $(1+x)\log(1+x) \ge x$, the left-hand side exceeds $px+2x$ which combined with the inequality $\log(1+x) \le x$ yields $px+2x \ge (p+2)\log(1+x) \ge 2p\log(1+x)$ as desired since $p\le 2$.
\end{proof}

\begin{proof}[Proof of Theorem~{\ref{thm:adiv}}]
  Nonnegativity of $\rdiv_p$ and $\rdiv_p(X,Y)=0$ iff $X=Y$ are immediate from~\eqref{eq:18}, while symmetry follows since $|X-Y|$ and $|Y-X|$ have the same singular values. Thus, only the triangle inequality requires an argument.

  The key idea is to reduce the question to a setting where we can apply a subadditivity theorem for singular values. For complex matrices $A$, $B$, $C$ of size $n\times n$ such that $C=A+B$, and any concave function $f: \reals_+\to [0,\infty)$ with $f(0)=0$, we have
  \begin{equation}
    \label{eq:19}
    \{f(\sigma_j(C))\}_{j=1}^n \prec_w \{f(\sigma_j(A))+ f(\sigma_j(B))\}_{j=1}^n,
  \end{equation}
  where $\prec_w$ denotes the weak-majorization partial order (see e.g., \citep[Chapter~2]{bhatia1997matrix}), and $\sigma_j(\cdot)$ denotes the $j$-th singular value (in decreasing order). The case of the full sum ($k=n$ in the partial sums defining $\prec_w$) is the classical inequality of \citet{rotfeld1969singular}; the weak-majorization form~\eqref{eq:19}, equivalently the same inequality in every unitarily invariant norm, is due to \citet{uchiyama2006subadditivity}, and it is this form that we need below.

  Let $A=X-Z$, $B=Z-Y$, and $C=X-Y$; let $a_j$, $b_j$, and $c_j$ be their singular values. Since $f(t)=\sqrt{\log(1+t^p)}$ is concave (see Lemma~\ref{lem:cve}), from~\eqref{eq:19} we thus obtain
  \begin{equation}
    \label{eq:20}
    \{f(c_j)\}_j \prec_w \{f(a_j) + f(b_j)\}_j.
  \end{equation}
  We further know that if $x \prec_w y$ for $x, y \in \reals_+^n$, then $x^s \prec_w y^s$ for $s \ge 1$ (see e.g.,~\citep[Example II.3.5]{bhatia1997matrix}). Consequently, for $s=2$ from inequality~\eqref{eq:20} we obtain
  \begin{equation}
    \label{eq:21}
    \{f(c_j)^2\}_j \prec_w \bigl(\{f(a_j) + f(b_j)\}_j\bigr)^2.
  \end{equation}
  Thus, in particular it follows from the majorization inequality~\eqref{eq:21} that
  \begin{equation}
    \label{eq:22}
    \nlsum_{j=1}^n \log(1+c_j^p) \le \nlsum_j \left(\sqrt{\log(1+a_j^p)} + \sqrt{\log(1+b_j^p)} \right)^2,
  \end{equation}
  so that upon taking square roots on both sides we obtain
  \begin{equation*}
    \sqrt{\nlsum_{j=1}^n \log(1+c_j^p)} \le \sqrt{\nlsum_j \left(\sqrt{\log(1+a_j^p)} + \sqrt{\log(1+b_j^p)} \right)^2}.
  \end{equation*}
  Applying Minkowski's inequality on the right hand side we get
  \begin{equation}
    \label{eq:23}
    \sqrt{\nlsum_{j=1}^n \log(1+c_j^p)} \le \sqrt{\nlsum_j \log(1+a_j^p)} + \sqrt{\nlsum_j \log(1+b_j^p)}.
  \end{equation}
  But $\nlsum_j\log(1+c_j^p)=\log\det(I+|C|^p)=\log\det(I+|X-Y|^p)=\rdiv_p^2(X,Y)$. Thus, inequality~\eqref{eq:23} is nothing but the triangle inequality
  \begin{equation*}
    \rdiv_p(X,Y) \le \rdiv_p(X,Z) + \rdiv_p(Y,Z),
  \end{equation*}
  which concludes the proof.
\end{proof}

\end{document}